\documentclass[12pt,a4paper]{article}
\usepackage{amsmath, amssymb,amsfonts,amsthm, amscd}
\usepackage{graphicx, subcaption}
\usepackage{lineno, hyperref}

\theoremstyle{plain}
\newtheorem{theorem}{Theorem}[section]
\numberwithin{equation}{section}

\newtheorem{corollary}{Corollary}
\newtheorem{example}{Example}
\newtheorem{remark}{Remark}
\theoremstyle{definition}

\DeclareMathOperator*{\essinf}{ess\,inf}
\DeclareMathOperator*{\esssup}{ess\,sup}

\title{On Ambarzumyan-type Inverse Problems of Vibrating String Equations}
\author{Yuri Ashrafyan\thanks{Yuri Ashrafyan: KAUST, CEMSE Division, Thuwal, 23955-6900, KSA, E-mail: yuri.ashrafyan@kaust.edu.sa}
\ and \ Dominik L.~Michels\thanks{ Dominik L.~Michels: KAUST, CEMSE Division, Thuwal, 23955-6900, KSA, E-mail: dominik.michels@kaust.edu.sa}
}
\date{}

\begin{document}

\maketitle

\begin{abstract}
We consider the inverse spectral theory of vibrating string equations. In this regard, first eigenvalue Ambarzumyan-type uniqueness theorems are stated and proved subject to separated, self-adjoint boundary conditions. More precisely, it is shown that there is a curve in the boundary parameters' domain on which no analog of it is possible. Necessary conditions of the $n$-th eigenvalue are identified, which allows to state the theorems. In addition, several properties of the first eigenvalue are examined. Lower and upper bounds are identified, and the areas are described in the boundary parameters' domain on which the sign of the first eigenvalue remains unchanged. This paper contributes to inverse spectral theory as well as to direct spectral theory.
\end{abstract}

\textit{Keywords:}\; Ambarzumyan theorem, first eigenvalue, inverse problems, vibrating string equations.

\textit{MSC 2010:}\; 34A55, 34L15

\vspace{5mm}

\section{Introduction}\label{sec1}
When dealing with direct problems, one considers a physical model and calculates a specific output given a specific input. In contrast, inverse problems are dealing with the inversion of this model based on measured or observed outputs, i.e.~we consider a mathematical framework which is used to obtain information about parameters of a system from observed measurements. This framework is often useful since it provides information about physical parameters of a system that can usually not directly be observed. As a consequence, inverse problem theory is being developed intensively.

The inverse spectral problems aiming for the reconstruction of operator using spectral data, such as the spectrum (set of eigenvalues), norming constants ($L^2$-norms of eigenfunctions), nodal points (roots of eigenfunctions), and other quantities.

Historically, Ambarzumyan did pioneer work \cite{Ambarzumyan:1929} with respect to the theory of inverse problems. 
It is easy to calculate (direct problem) that the eigenvalues of the Sturm-Liouville boundary value problem $-y'' = \mu y$, $y'(0) = y'(\pi) = 0$ are $n^2, \ n \geq 0$.
In 1929, Ambarzumyan proved the inverse assertion, if the eigenvalues of the boundary value problem $-y'' + q(x) y = \mu y$, $y'(0) = y'(\pi) = 0$ are $n^2, \ n \geq 0$, then $q(x) = 0$ a.e.~on $(0, \pi)$.

Later, Borg \cite{Borg:1946} showed that with general boundary conditions more information in addition to the spectrum is required in order to uniquely reconstruct the function $q(x)$.
In the same work he showed that two spectra (for different fixed boundary conditions) are sufficient for the unique determination of a problem.
Other fundamental work was done by Marchenko \cite{Marchenko:1950}, who showed that if two problems have the same set of eigenvalues and norming constants, then they coincide.
These results were very important and they opened the door for further investigations in inverse spectral theory.
Many other kinds of uniqueness theorems for inverse Sturm-Liouville problems have been stated and developed, we just mention several of them here~ \cite{Levinson:1949, Hochstadt-Lieberman:1978, McLaughlin-Rundell:1987, Gesztesy-Simon:2000, Freiling-Yurko:2001, Harutyunyan:2009, Wei-Wei:2016, Ashrafyan:2017}.

Let $L(q, \gamma, \delta)$ denote the self-adjoint Sturm-Liouville operator generated by the boundary value problem
\begin{gather}
-y''+q(s)y = \mu y,\qquad s\in (0, \pi),\; \mu \in {\mathbb C},         \label{eq1.4}\\
y(0) \cos \gamma + y'(0) \sin \gamma = 0, \quad \gamma \in (0, \pi],    \label{eq1.5}\\
y(\pi) \cos \delta + y'(\pi) \sin \delta = 0, \quad \delta \in[0, \pi), \label{eq1.6}
\end{gather}
where the potential function $q$ is real-valued and summable.
It is known, that the spectrum of $L(q, \gamma, \delta)$ is discrete and consists of real, simple eigenvalues, which we denote by $\mu_n = \mu_n(q, \gamma, \delta)$, $n \geq 0$,
$$\mu_0(q, \gamma, \delta) < \mu_1(q, \gamma, \delta) < \ldots < \mu_n(q, \gamma, \delta) < \ldots,$$
emphasizing the dependence of $\mu_n$ on $q$, $\gamma$ and $\delta$ (see \cite{Levitan-Sargsyan:1988, Marchenko:1977, Harutyunyan:2008}).
There are many generalizations of Ambarzumyan's theorem for the Sturm-Liouville problem in various directions ~(see e.g.~\cite{Freiling-Yurko:2001, Kuznezov:1962, Chakravarty-Acharyya:1988, Freiling-Yurko:2001, Yurko:2013, Chern-Law-Wang:2001, Yang-Huang-Yang:2010, Yilmaz-Koyunbakan:2014, Ashrafyan:2018}, and the references therein).

Particularly, in \cite{Freiling-Yurko:2001}, it is shown that it is not necessary to specify the whole spectrum: information regarding the first eigenvalue is sufficient.
Precisely, if $\mu_0 = \frac{1}{\pi}\int_{0}^{\pi} q(x) dx$, then $q(x) = \mu_0$, a.e.~on $(0, \pi)$.
This type of problem setting is called a first eigenvalue Ambarzumyan-type inverse problem.

Let $A$ and $\tilde A$ be two operators.
In what follows, if a certain symbol $\tau$ denotes an object related to the operator $A$, then $\tilde \tau$ will denote a similar object related to $\tilde A$ and $\hat \tau := \tilde \tau - \tau$, $\check \tau := {\tilde \tau}/{\tau}$.
We recall the inner product $\left(y, z\right) = \int_{0}^{1} y(x) \overline{z(x)} dx$.

In \cite{Yurko:2013}, Yurko proved the following generalization of Ambarzumyan theorem for Sturm-Liouville problem (it's true for any self-adjoint boundary conditions).
\begin{theorem}[Yurko \cite{Yurko:2013}]\label{thm1.2}
  Let $q, \, \tilde q \in L^1_\mathbb{R}(0, \pi)$. If
  \begin{equation*}
    \mu_0 - \tilde \mu_0 = \dfrac{(\hat{q} \tilde \varphi_0, \tilde \varphi_0)}{(\tilde \varphi_0, \tilde \varphi_0)},
  \end{equation*}
  where $\tilde \varphi_0$ is an eigenfunction of $\tilde L$ related to $\tilde \mu_0$.
  Then $q(x) = \tilde q(x) + \mu_0 - \tilde \mu_0$, a.e.~on $(0, \pi)$.
\end{theorem}
Recently, one of the authors  \cite{Ashrafyan:2018} has proved another generalization of the first eigenvalue Ambarzumyan-type inverse problem.
\begin{theorem}[\cite{Ashrafyan:2018}]\label{thm1.3}
    Let $q, \, \tilde q \in L^1_\mathbb{R}(0, \pi)$. If
    \begin{equation*}
      \mu_0 - \tilde \mu_0 = \essinf \hat q  \quad \mbox{or} \quad \mu_0 - \tilde \mu_0 = \esssup \hat q,
    \end{equation*}
    then $q(x) = \tilde q(x) + \mu_0 - \tilde \mu_0$ a.e.~on $(0, \pi)$.
\end{theorem}

The main aim of this paper is to investigate first eigenvalue Ambarzumyan-type theorems for boundary value problems of vibrating string equations:
\begin{gather}
- u'' = \lambda p(x) u, \quad  x \in (0, 1), \quad \lambda \in \mathbb{C},  \label{eq1.1} \\
u(0) \cos \alpha + u'(0) \sin \alpha = 0, \quad \alpha \in (0, \pi],        \label{eq1.2} \\
u(1) \cos \beta + u'(1) \sin \beta = 0, \quad \beta \in [0, \pi),           \label{eq1.3}
\end{gather}
where the density function $p$ is piecewise continuous, positive, bounded away from $0$, real-valued and $\lambda$ is the spectral parameter.
This equation describes small transversal vibrations of a string of a linear density $p$ on the interval $[0, 1]$ (see e.g.~\cite{Wei-Xu:2010}).
By $S (p, \alpha, \beta)$ we denote the self-adjoint operator generated by problem \eqref{eq1.1}--\eqref{eq1.3}.
It is well-known, that the spectrum of the operator $S(p, \alpha, \beta)$ is discrete and consists of real, simple eigenvalues, which we denote by $\lambda_n = \lambda_n(p, \alpha, \beta)$, $n \geq 0$,
$$\lambda_0(p, \alpha, \beta) < \lambda_1(p, \alpha, \beta) < \ldots < \lambda_n(p, \alpha, \beta) < \ldots,$$
emphasizing the dependence of $\lambda_n$ on $p$, $\alpha$ and $\beta$.

The classical Ambarzumyan's theorem for the operator $S(p, \pi/2, \pi/2)$ on the interval $[0, a]$ can be stated as follows: if it is known that the eigenvalues are $\pi^2 n^2/ a^2$, the question is if we can conclude that $p(x)=1$ on $[0, a]$.
In \cite{Shen:2007}, Shen constructed a counterexample to this assertion, and showed that one needs additional information about the density function to formulate an inverse problem with one spectrum, such as the function $p(x)$ being even on $[0, a]$, or the values of $p'(0)$ and $p'(a)$ are known.
Hence, in general, there is no analog to the Ambarzumyan theorem for string equations in a classical form.

In Section \ref{sec2}, it is shown that first eigenvalue Ambarzumyan-type theorems for the $S(p, \alpha, \beta)$ operator are valid, but not for the whole domain $\alpha, \beta \in (0, \pi] \times [0, \pi)$. It is explained why for some boundary conditions first-eigenvalue inverse problems are impossible.
In Section \ref{sec3}, bounds for the lowest eigenvalue are given, and the areas in the boundary parameters’ domain are described, on which it keeps its sign.

\section{Main Results}\label{sec2}
Let $\varphi(x) = \varphi(x,\lambda, \alpha,p)$ be the solution of \eqref{eq1.1}, which satisfies the initial conditions
\begin{equation*}
\varphi(0,\lambda, \alpha,p)=\sin  \alpha,\quad \varphi'(0,\lambda, \alpha,p)=-\cos  \alpha. \\
\end{equation*}
The eigenvalues $\lambda_n$ are the solutions of the equation
\begin{equation*}
\varphi(1, \lambda,  \alpha)\cos  \beta + \varphi'(1, \lambda, \alpha)\sin \beta=0.
\end{equation*}
It is easy to see that the functions $\varphi_n(x):=\varphi(x, \lambda_n,  \alpha, p)$, $n \geq 0$, are the eigenfunctions, corresponding to the eigenvalue $\lambda_n$.

Our first eigenvalue Ambarzumyan-type theorems for vibrating string equation are as follows.
\begin{theorem}\label{thm2.1}
    Let $\alpha, \beta \in (0, \pi] \times [0, \pi)$ and $\cos \alpha \cos \beta - \sin(\alpha - \beta) \neq 0$, for $\alpha, \beta \in [\pi/4, \pi] \times [0, 3\pi/4]$.
    If
    \begin{equation*}
      \lambda_0 = \tilde \lambda_0  \max_{x \in [0, 1]}  \check p(x)  \qquad \mbox{or} \qquad
      \lambda_0 = \tilde \lambda_0  \min_{x \in [0, 1]}  \check p(x),
    \end{equation*}
    then $p(x) = {\tilde \lambda_0}/{\lambda_0} \tilde p(x)$ on $[0, 1]$.
\end{theorem}

\begin{proof}
Write down the fact that $\varphi_0(x)$ and $\tilde\varphi_0(x)$ are the eigenfunctions of the operators $S(p, \alpha, \beta)$ and $S(\tilde p, \alpha, \beta)$ corresponding to the eigenvalues $\lambda_0$ and $\tilde \lambda_0$, respectively:
\begin{gather}
  -\varphi_0''(x) + \lambda_0 p(x) \varphi_0(x) =  0, \label{eq2.1} \\
  -\tilde \varphi_0''(x) + \tilde \lambda_0 \tilde p(x) \tilde \varphi_0(x) = 0. \label{eq2.2}
\end{gather}
Multiplying \eqref{eq2.1} by $\tilde \varphi_0(x)$, \eqref{eq2.2} by $\varphi(x)$ and subtracting from the first equation the second one, we obtain
\begin{equation*}
\begin{aligned}
  - \dfrac{d}{dx}\left[ \varphi_0'(x) \tilde \varphi_0(x) - \varphi_0(x)\tilde \varphi_0'(x)\right] + \\
  + \varphi_0(x) \tilde \varphi_0(x)\left[\lambda_0 p(x) - \tilde \lambda_0 \tilde p(x)\right] = 0.
\end{aligned}
\end{equation*}
Integrating the latter from $0$ to $1$, we obtain
\begin{equation*}
\begin{aligned}
    - \varphi_0'(1) \tilde \varphi_0(1) + \varphi_0(1) \tilde \varphi_0'(1) + \\
    + \varphi_0'(0) \tilde \varphi_0(0) - \varphi_0(0) \tilde \varphi_0'(0) + \\
    + \int_{0}^{1} \varphi_0(x) \tilde \varphi_0(x) \left[\lambda_0 p(x) - \tilde \lambda_0 \tilde p(x) \right] dx = 0.
\end{aligned}
\end{equation*}
Since $\varphi_0(x)$ and $\tilde \varphi_0(x)$ satisfy the same boundary conditions, thus the first four terms vanish, and we obtain
\begin{equation*}
  \int_{0}^{1} \varphi_0(x) \tilde \varphi_0(x) \left[\lambda_0 p(x) - \tilde \lambda_0 \tilde p (x)\right] dx = 0.
\end{equation*}
In virtue of Sturm's oscillation theorem (see e.g.~\cite{Freiling-Yurko:2001, Hinton:2005, Simon:2005, Harutyunyan-Pahlevanyan-Ashrafyan:2013}), the product $\varphi_0(x) \tilde \varphi_0(x)$ has no zeros in interval $(0, 1)$.
This with the condition of the theorem complete the proof.
\end{proof}

\begin{theorem}\label{thm2.2}
    Let $\alpha, \beta \in (0, \pi] \times [0, \pi)$ and $\cos \alpha \cos \beta - \sin(\alpha - \beta) \neq 0$, for $\alpha, \beta \in [\pi/4, \pi] \times [0, 3\pi/4]$.
    If
    \begin{equation*}
        \lambda_0 = \tilde \lambda_0 \dfrac{\left( \check p \tilde \varphi_0, \tilde \varphi_0 \right)}{\left( \tilde \varphi_0, \tilde \varphi_0 \right)}
    \end{equation*}
    then $p(x) = {\tilde \lambda_0}/{\lambda_0} \tilde p(x)$ on $[0, 1]$.
\end{theorem}

\begin{proof}
Taking into account the condition of the theorem, these simple relations are held
\begin{equation*}
\dfrac{\left( - p^{-1} \tilde \varphi_0, \tilde \varphi_0 \right)}{\left(\tilde \varphi_0, \tilde \varphi_0 \right)} =
\dfrac{\left( - \tilde p^{-1} \tilde \varphi_0, \check p \tilde \varphi_0 \right)}{\left(\tilde \varphi_0, \tilde \varphi_0 \right)} =
\tilde \lambda_0 \dfrac{\left(\check p \tilde \varphi_0, \tilde \varphi_0 \right)}{\left(\tilde \varphi_0, \tilde \varphi_0 \right)} =
\lambda_0
\end{equation*}
from which follows that $\tilde \varphi_0(x)$ is an eigenfunction of the operator $S(p, \alpha, \beta)$ with the eigenvalue $\lambda_0$.
Particularly, $-\tilde \varphi_0''(x) = \lambda_0 p(x) \tilde \varphi_0(x)$, thus we obtain $\lambda_0 p(x) = \tilde \lambda_0 \tilde p(x)$.
This completes the proof.
\end{proof}

\begin{remark}\label{rem1}
   Consider the curve $\cos \alpha \cos \beta - \sin(\alpha - \beta) = 0$, for $\alpha, \beta \in [\pi/4, \pi] \times [0, 3\pi/4]$, see Figure \ref{fig:zeros}.
   The first eigenvalue $\lambda_0(p,  \alpha, \beta)$ on this curve is always zero (the proof of this assertion is given in Theorem \ref{thm3.3}). Therefore first-eigenvalue Ambarzumyan-type theorems are impossible for this $\alpha$ and $\beta$.
\end{remark}

\begin{figure}[!ht]
\center
  \includegraphics[scale=0.5]{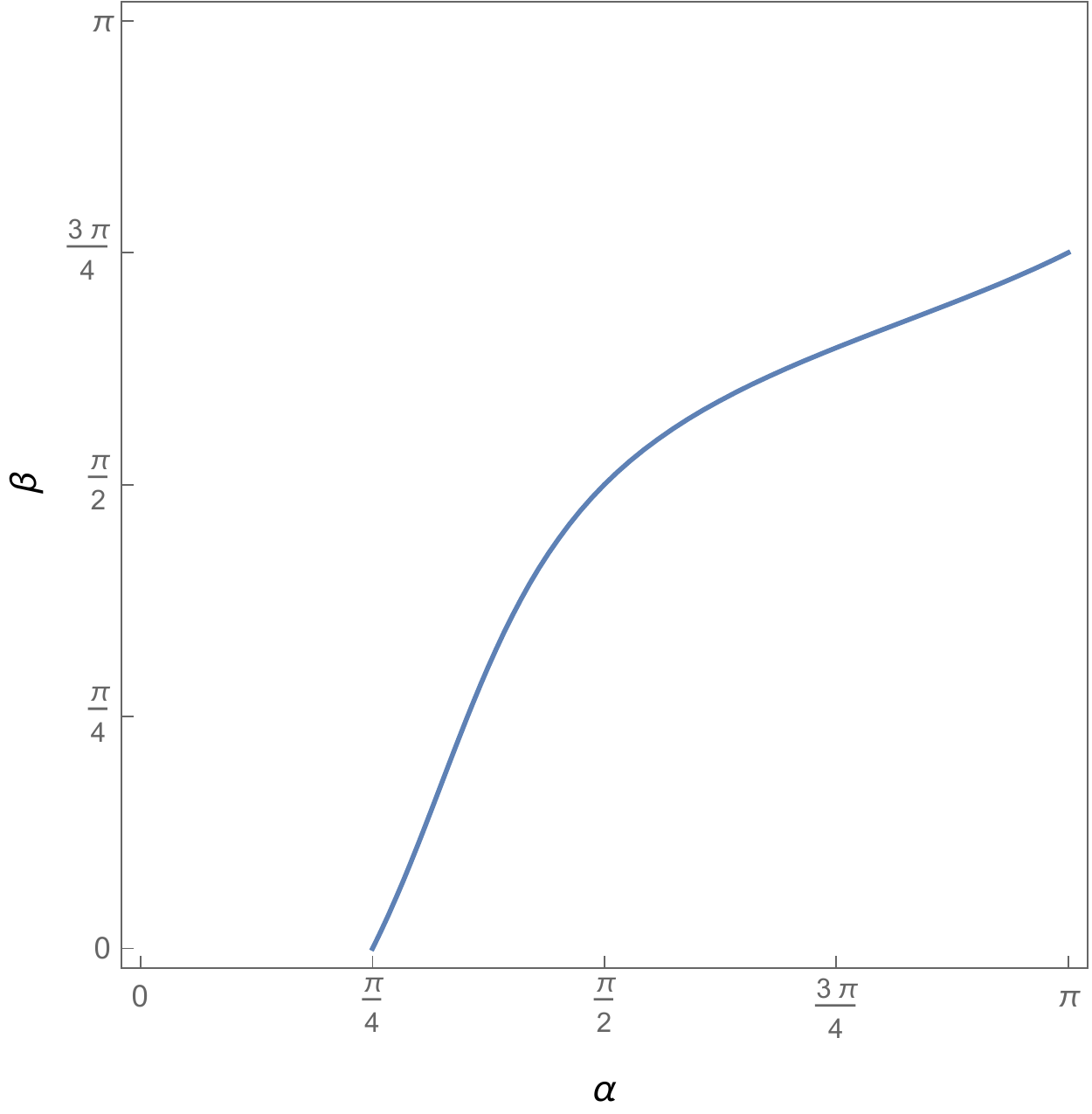}
  \caption{Illustration of the curve describing $\lambda_0(p,  \alpha, \beta)=0$.}
  \label{fig:zeros}
\endcenter
\end{figure}

\begin{example}
Consider the boundary value problem 
\begin{gather*}
- u'' = \lambda p(x) u,\\
u(0) = u(1) = 0.
\end{gather*}
Let $\tilde p(x) = 1$, then $\tilde \lambda_0 = \pi^2$ and $\tilde \varphi_0(x) = \sin \pi x$.

Theorem \ref{thm2.1} implies the following assertion.
\begin{corollary}
If
\begin{equation*}
    \lambda_0(p, \pi, 0) = \pi^2 \max p^{-1}(x)
\end{equation*}
then $p(x) = \pi^2/ \lambda_0$ on $[0, 1]$.
\end{corollary}

Theorem \ref{thm2.2} implies the following assertion.
\begin{corollary}
If
\begin{equation*}
    \lambda_0(p, \pi, 0) = 2 \pi^2 \int_{0}^{1} p^{-1}(x) sin^2(\pi x) dx,
\end{equation*}
then $p(x) = \pi^2/ \lambda_0$ on $[0, 1]$.
\end{corollary}
\end{example}

As we saw that the lowest eigenvalue contains significant information about the equation, a reasonable question may arise: can one state a uniqueness theorem having information regarding an arbitrary eigenvalue?
The answer is positive, but more information is demanded, than in the case of the lowest eigenvalue.
\begin{theorem}\label{thm2.3}
    Let $\lambda_n \neq 0$, for a fixed $n > 0$.
    If
    \begin{equation*}
        \lambda_n = \tilde \lambda_n \dfrac{\left( \check p \tilde \varphi_n, \tilde \varphi_n \right)}{\left( \tilde \varphi_n, \tilde \varphi_n \right)}
        \quad \mbox{and} \quad
        \lambda_n = \tilde \lambda_n  \max_{x \in [0, 1]}  \check p(x),
    \end{equation*}
    then $p(x) = {\tilde \lambda_n}/{\lambda_n} \tilde p(x)$ on $[0, 1]$.
\end{theorem}

\begin{proof}
From the first condition we obtain
\begin{equation*}
\int_{0}^{1} ( \lambda_n - \tilde \lambda_n \check p(x) ) \tilde \varphi_n^2(x) dx = 0.
\end{equation*}
In virtue of Sturm's oscillation theorem, it is known that the $n$-th eigenfunction has exactly $n$ isolated zeros in the open interval $(0, 1)$.
Thus, the measure of the set of all zeros is $0$.
This with the second condition of the theorem yields $\lambda_n - \tilde \lambda_n \check p(x) = 0$.
This completes the proof.
\end{proof}

For Sturm-Liouville problem similar result was proved in \cite{Kirac:2020}.
\begin{remark}
  It is obvious that in Theorem \ref{thm2.3} the second condition can be replaced with $\lambda_n = \tilde \lambda_n  \min_{x \in [0, 1]}  \check p(x)$.
\end{remark}

\section{Exploring the Lowest Eigenvalue}\label{sec3}
For the case $\sin \gamma \neq 0$ and $\sin \delta \neq 0$, i.e.~$\gamma, \delta \in (0, \pi)$, the boundary conditions \eqref{eq1.5}--\eqref{eq1.6} can be written as
\begin{gather*}
a \, y(0) + y'(0) = 0, \\
b \, y(\pi) + y'(\pi) = 0.
\end{gather*}
In \cite{Isaacson-McKean-Trubowitz:1984} the authors showed, that for a Sturm-Liouville problem with fixed boundary conditions $(a, b)$ and a $q \in L^2(0, \pi)$, the lowest eigenvalue has the property
\begin{equation*}
  -\infty < \mu_0(q, a, b) \leq \mu_0 \left( e, - (a-b)/2, (a-b)/2  \right),
\end{equation*}
where $e$ is a known even function ($e(x) = e(\pi-x)$) from $L^2(0, \pi)$,  which is uniquely defined by the eigenvalues $\mu_n(q, a, b), n \geq 1$ and the Dirichlet boundary conditions $y(0) = y(\pi) = 0$.

In \cite{Poschel-Trubowitz:1987}, for a $q \in L^\infty(0, \pi)$ and the Dirichlet boundary conditions, the following bounds for the lowest eigenvalue is proved:
\begin{equation}\label{eq3.1}
  | \mu_0(q, \pi, 0) - \mu_0(0, \pi, 0) | \leq \| q \|_\infty.
\end{equation}

In \cite{Harutyunyan:2008}, for a $q$ from $L^1_\mathbb{R}(0, \pi)$ (in particular $L^\infty_\mathbb{R}(0, \pi)$) and $\gamma \in (0, \pi]$,  $\delta \in [0, \pi)$, a formula for $\mu_n(q, \gamma, \delta)$ is found, and particularly for $n = 0$:
\begin{equation}\label{eq3.2}
  \mu_0(q, \gamma, \delta) = \mu_0(0, \gamma, \delta) +
  \displaystyle \int_{0}^{1} \int_{0}^{\pi} q(x) h^2_0(x, tq, \gamma, \delta) \, dx \, dt,
\end{equation}
where $h_0(x, tq, \gamma, \delta)$ is a normalized eigenfunction ($\int_{0}^{\pi} h^2_0(x, tq, \gamma, \delta) dx = 1 $) of the problem $L(tq, \gamma, \delta)$ and where $t$ is a real parameter.
From \eqref{eq3.2} the generalization of \eqref{eq3.1} follows:
\begin{equation*}
  | \mu_0(q, \gamma, \delta) - \mu_0(0, \gamma, \delta) | \leq \| q \|_\infty,
\end{equation*}
which can be written as
\begin{equation}\label{eq3.3}
  -\esssup |q| \leq \mu_0(q, \gamma, \delta) - \mu_0(0, \gamma, \delta)  \leq \esssup |q|.
\end{equation}
In the same work it is shown that the lowest eigenvalue tends to $-\infty$:
\begin{gather*}
    \displaystyle \lim_{\gamma \rightarrow 0} \mu_0(q, \gamma, \delta) = - \infty, \qquad \mbox{for fixed} \; \delta \in [0, \pi), \\
    \displaystyle \lim_{\delta \rightarrow \pi} \mu_0(q, \gamma, \delta) = - \infty, \qquad \mbox{for fixed} \; \gamma \in (0, \pi].
\end{gather*}

We found narrower bounds for the lowest eigenvalue of the Sturm-Liouville operator.
\begin{theorem}\label{thm3.1}
    The lowest eigenvalue of the operator $L(q, \gamma, \delta)$ has the property
    \begin{equation*}
      \essinf q  \leq  \mu_0(q, \gamma, \delta) - \mu_0(0, \gamma, \delta) \leq  \esssup q.
    \end{equation*}
\end{theorem}
\begin{remark}
    Theorem \ref{thm3.1} is also true for arbitrary self-adjoint boundary conditions.
\end{remark}
Theorem \ref{thm3.1} is a corollary of Theorem \ref{thm1.3}, when we take $\tilde q(x) \equiv 0$.
We can see that our bounds do not contradict the properties mentioned above and are more narrow than in \eqref{eq3.3} because of
\begin{equation*}
  - \esssup |q| \, \leq \, \essinf q \, \leq \, \esssup q \, \leq \, \esssup |q|.
\end{equation*}

\begin{figure}[!ht]
  \center
  \begin{subfigure}[b]{0.4\linewidth}
    \includegraphics[width=\linewidth]{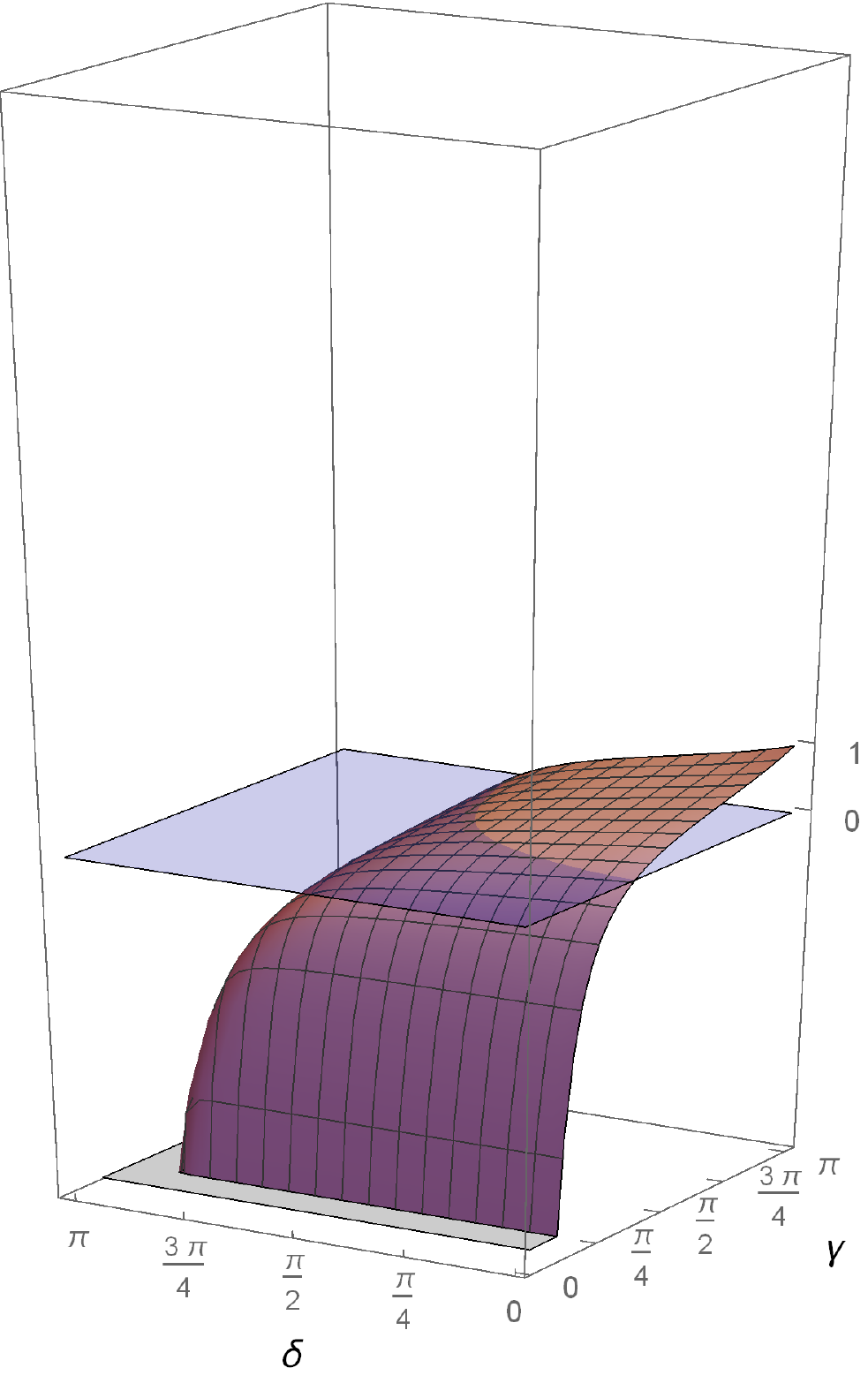}
    \caption{$\mu_0(0, \gamma, \delta)$}
  \end{subfigure}
  \begin{subfigure}[b]{0.42\linewidth}
    \includegraphics[width=\linewidth]{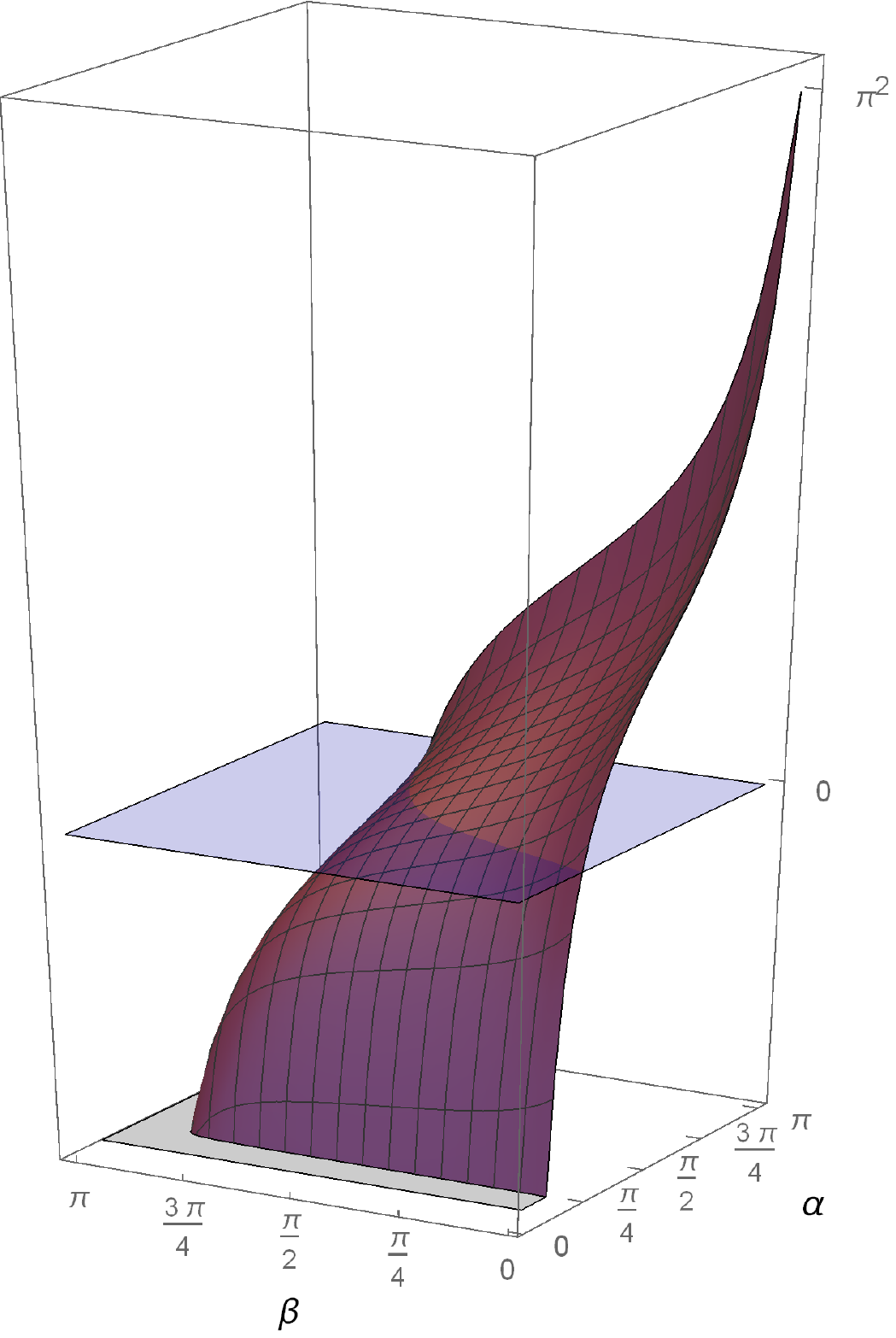}
    \caption{$\lambda_0(1, \alpha, \beta)$}
  \end{subfigure}
  \caption{Illustration of the first eigenvalues of $L(0, \gamma, \delta)$ and $S(1, \alpha, \beta)$.}
  \label{fig:first}
  \endcenter
\end{figure}

For the first eigenvalue of the vibrating string equation we obtain the following bounds.
\begin{theorem}\label{thm3.2}
    The lowest eigenvalue of operator $S(p, \alpha, \beta)$ has the property
    \begin{equation*}
      \lambda_0(1, \alpha, \beta) \min_{x \in [0,\pi]} p^{-1}(x) \leq
      \lambda_0(p, \alpha, \beta) \leq
      \lambda_0(1, \alpha, \beta) \max_{x \in [0,\pi]} p^{-1}(x).
    \end{equation*}
\end{theorem}

\begin{proof}
 This is a corollary of Theorem \ref{thm2.1}.
\end{proof}

The graphs of eigenvalues $\mu_0(0, \gamma, \delta)$ and $\lambda_0(1, \alpha, \beta)$ are in Figure \ref{fig:first}.
The relation between eigenvalues of $L(q, \gamma, \delta)$ and $S(p,  \alpha, \beta)$ can be found in the appendix.

\begin{example}
Consider the boundary value problem $S(x^r+1, \pi, 0)$, for a fixed $r \in \mathbb{R}_+$,
\begin{gather*}
- u'' = \lambda (x^r+1) u, \\
u(0) = u(1) = 0.
\end{gather*}
Then $ \lambda_0(1, \pi, 0) = \pi^2$, and $\displaystyle \max_{x \in [0,\pi]} (x^r+1)^{-1} = 1$, $\displaystyle \min_{x \in [0,\pi]} (x^r+1)^{-1} = 1/2$.
Theorem \ref{thm3.2} implies the following assertion.
\begin{corollary}
For any fixed $r \in \mathbb{R}_+$:
    \begin{equation*}
      \pi^2 /2 \leq
      \lambda_0(x^r+1, \pi, 0) \leq
      \pi^2.
    \end{equation*}
For instance, when r=1,2, the lowest eigenvalues are $\lambda_0(x+1, \pi, 0) \approx 6.548$ and $\lambda_0(x^2+1, \pi, 0) \approx 7.643$.
\end{corollary}
\end{example}

The next theorem shows, that for any density function $p(x)$ the sign of the first eigenvalue $\lambda_0(p, \alpha, \beta)$ only depends on the boundary parameters.
\begin{theorem}\label{thm3.3}
    The lowest eigenvalue of the operator $S(p, \alpha, \beta)$ has the property
    \begin{equation*}
        \lambda_0(p, \alpha, \beta) \left\{
            \begin{array}{lll}
                =0, & \cos \alpha \cos \beta - \sin (\alpha - \beta)=0, & \alpha, \beta \in [\pi/4, \pi] \times [0, 3 \pi/4],\\
                >0, & \cos \alpha \cos \beta - \sin (\alpha - \beta)<0, & \alpha, \beta \in (\pi/4, \pi] \times [0, 3 \pi/4),\\
                <0, & \mbox{otherwise}.
            \end{array}
            \right.
    \end{equation*}
\end{theorem}

\begin{proof}
 First, we will show that the lowest eigenvalue  $\lambda_0$ is zero for boundary parameters satisfying
\begin{equation}\label{eq4.1}
  \cos \alpha \cos \beta - \sin(\alpha - \beta) = 0, \qquad \alpha, \beta \in [\pi/4, \pi] \times [0, 3\pi/4].
\end{equation}
Let us check that $\lambda = 0$ is an eigenvalue.
When $\lambda = 0$ the equation \eqref{eq1.1} becomes
\begin{equation*}
  u''(x) = 0,
\end{equation*}
hence, in order for $u(x)$ to be an eigenfunction of $S(p,  \alpha, \beta)$ related to an eigenvalue $\lambda = 0$, it should be linear $u(x) = k x + c$, where $k^2 + c^2 \neq 0$, and satisfy boundary conditions, \eqref{eq1.2}--\eqref{eq1.3}
\begin{gather*}
u(0) \cos \alpha + u'(0) \sin \alpha = 0, \\
u(1) \cos \beta + u'(1) \sin \beta = 0.
\end{gather*}
For $u(x) = k x + c$, we consider the system of linear equations
\begin{gather*}
c \cos \alpha + k \sin \alpha = 0, \qquad  \qquad \\
c \cos \beta + k (\sin \beta + \cos \beta) = 0.
\end{gather*}
In order to obtain a unique solution for $k$ and $c$, the determinant should be zero:
\begin{equation*}
\left|
\begin{array}{cc}
  \cos \alpha & \sin \alpha \\
  \cos \beta & \sin \beta + \cos \beta
\end{array}
\right|
=
\cos \alpha \sin \beta + \cos \alpha \cos \beta - \sin \alpha \cos \beta
\end{equation*}
which is zero by the condition.
Therefore, $\lambda = 0$ is an eigenvalue, and it remains to show that this is the smallest eigenvalue.
Let there be a number $n_0 > 0$ such that $\lambda_{n_0} = 0$.
This means that there exists $\sigma \neq 0$, such that the smallest eigenvalue $\lambda_0 = - \sigma^2$.
Thus, the equation \eqref{eq1.1} can be rewritten as
\begin{equation}\label{eq4.0}
  u''_0(x) = \sigma^2 p(x) u_0(x).
\end{equation}

Now we separate the curve into following three cases:
\begin{enumerate}
  \item $\alpha \in [\pi/4, \pi/2), \beta \in [0, \pi/2)$,
  \item $\alpha = \beta = \pi/2$,
  \item $\alpha \in (\pi/2, \pi], \beta \in (\pi/2, 3\pi/4]$.
\end{enumerate}
We will only show the proof for the first case since the other two cases are completely similar.
In this case the boundary conditions become
\begin{gather*}
u_0(0) + u'_0(0) (\tan \beta + 1) = 0,\\
u_0(1) + u'_0(1) \tan \beta = 0.
\end{gather*}
Subtracting from the second condition the first one, we obtain
\begin{equation*}
  (u'_0(1) - u'_0(0)) \tan \beta + (u_0(1) - u_0(0) - u'_0(0)) = 0.
\end{equation*}
For simplicity, it can be written as
\begin{equation*}
  a \tan \beta + b = 0.
\end{equation*}
To find $a$ and $b$, we integrate \eqref{eq4.0} from $0$ to $0<t\leq1$ and afterwards, from $0$ to $1$:
\begin{gather*}
u'_0(t) - u'_0(0)= \sigma^2 \int_{0}^{t} p(x) u_0(x) dx, \\
u'_0(1) - u'_0(0)= \sigma^2 \int_{0}^{1} p(x) u_0(x) dx = a, \\
u_0(1) - u_0(0) - u'(0)= \sigma^2 \int_{0}^{1} \int_{0}^{t} p(x) u_0(x) dx dt = b.
\end{gather*}
Since the density function $p(x)$ is positive and $u_0(x)$ has no zeros, due to the oscillation theorem, in the whole interval $(0, 1)$, therefore $a \neq 0$, $b \neq 0$ and they have the same sign.
Now, please note that $\tan \beta \geq 0$, for $\beta \in [0, \pi/2)$, so the relation $a \tan \beta + b = 0$ is impossible.
Thus, the first eigenvalue cannot be negative and $\lambda_0 = 0$.

Here we proof the second assertion of the theorem, that is $\lambda_0(p, \alpha, \beta) > 0$, when $\alpha, \beta \in [\pi/4, \pi] \times [0, 3 \pi/4]$ and
\begin{equation}\label{eq4.2}
  \cos \alpha \cos \beta - \sin (\alpha - \beta)<0.
\end{equation}
The third one of the theorem can be handled in a similar way.
Please note, that we can show this for $\lambda_0(1, \alpha, \beta)$, after it, in virtue of Theorem \ref{thm3.2}, it will be spread on $\lambda_0(p, \alpha, \beta)$.

Assume that $\lambda_0(1, \alpha, \beta) < 0$, i.e.~there exists $\sigma \neq 0$, such that $\lambda_0 = - \sigma^2$.
So, the boundary value problem \eqref{eq1.1}--\ref{eq1.3} for $p(x)=1$ is being written as
\begin{gather*}
u''_0(x) = \sigma^2 u_0(x), \\
u_0(0) \cos \alpha + u'_0(0) \sin \alpha  = 0,\\
u_0(1) \cos \beta + u'_0(1) \sin \beta = 0.
\end{gather*}
The solution of this problem has the following form:
\begin{equation*}
    u(x) = A e^{|\sigma|x} + B e^{-|\sigma|x},
\end{equation*}
where $A^2 + B^2 \neq 0$.
Inserting this into the boundary conditions, we obtain
\begin{gather*}
A (\cos \alpha + |\sigma|\sin \alpha)  + B (\cos \alpha - |\sigma| \sin \alpha) = 0, \qquad  \qquad \\
A e^{|\sigma|} (\cos \beta + |\sigma|\sin \beta) + B e^{-|\sigma|} (\cos \beta - |\sigma| \sin \beta) = 0.
\end{gather*}
To have a unique solution for $A$ and $B$, the determinant should be zero:
\begin{gather*}
D :=
\left|
\begin{array}{cc}
  \cos \alpha + |\sigma|\sin \alpha & \cos \alpha - |\sigma| \sin \alpha \\
  e^{|\sigma|} (\cos \beta + |\sigma|\sin \beta) & e^{-|\sigma|} (\cos \beta - |\sigma| \sin \beta)
\end{array}
\right|
= \\
=
\cos \alpha \cos \beta (e^{-|\sigma|}-e^{|\sigma|}) +
\sin (\alpha - \beta) |\sigma| (e^{-|\sigma|}+e^{|\sigma|}) +
\sin \alpha \sin \beta \sigma^2  (e^{|\sigma|}-e^{-|\sigma|}) = \\
=
2 (- \cos \alpha \cos \beta \sinh |\sigma| +
\sin (\alpha - \beta) |\sigma| \cosh |\sigma| +
\sin \alpha \sin \beta \sigma^2  \sinh |\sigma|).
\end{gather*}
Taking into account \eqref{eq4.2}, for $D$ we obtain
\begin{equation*}
  D > 2 (\cos \alpha \cos \beta ( |\sigma| \cosh |\sigma| - \sinh |\sigma|)+
\sin \alpha \sin \beta \sigma^2  \sinh |\sigma|),
\end{equation*}
or
\begin{equation*}
  D > 2 (\sin (\alpha - \beta) ( |\sigma| \cosh |\sigma| - \sinh |\sigma|)+
\sin \alpha \sin \beta \sigma^2  \sinh |\sigma|).
\end{equation*}
Please note that $\sin \alpha \sin \beta \sigma^2  \sinh |\sigma| \geq 0$, for all $\alpha, \beta \in (\pi/4, \pi] \times [0, 3 \pi/4)$.
Let us show that $|\sigma| \cosh |\sigma| - \sinh |\sigma| > 0$.
Consider
\begin{equation}\label{eq4.3}
  f(x) := x \cosh x - \sinh x
\end{equation}
for $x>0$.
Differentiating \eqref{eq4.3}, we obtain
\begin{equation*}
  f'(x) = x \sinh x.
\end{equation*}
Since $f(0)=0$ and $f'(x)>0$, for all $x>0$, thus $f(x)>0$.

Now we separate the domain into the following three subdomains:
\begin{enumerate}
  \item $\alpha \in (\pi/4, \pi/2], \beta \in [0, \pi/2)$,
  \item $\alpha \in (\pi/2, \pi], \beta \in [0, \pi/2)$,
  \item $\alpha \in (\pi/2, \pi], \beta \in [\pi/2, 3\pi/4)$.
\end{enumerate}
For the first and the third case $\cos \alpha \cos \beta \geq 0$, and for the second case $\sin (\alpha - \beta) \geq 0$.
Thus $D > 0$ and we come to contradiction.
This completes the proof.
\end{proof}

\begin{remark}
 Since on the curve $\cos \alpha \cos \beta - \sin(\alpha - \beta) = 0$, for $\alpha, \beta \in [\pi/4, \pi] \times [0, 3\pi/4]$, the first eigenvalue is always zero, thus the inequality in Theorem \ref{thm3.2} holds also for this case.
\end{remark}

\section{Conclusion}\label{sec4}
This paper deals with the inverse spectral theory of vibrating string equations subject to separated, self-adjoint boundary conditions.
It is shown, that in spite of the fact that there is no classical Ambarzumyan's theorem analog for string equations, the first eigenvalue Ambarzumyan-type theorems are valid for some domain in the boundary conditions (Theorems \ref{thm2.1}--\ref{thm2.2} and the extension for the $n$-th eigenvalue in Theorem \ref{thm2.3}).

The curve in the boundary conditions' domain ($\alpha, \beta \in (0, \pi] \times [0, \pi)$), on which no analog of the first eigenvalue theorem is possible, is given by the equation $\cos \alpha \cos \beta - \sin(\alpha - \beta) = 0$, for $\alpha, \beta \in [\pi/4, \pi] \times [0, 3\pi/4]$ (see Figure \ref{fig:zeros}).
The explanation simply follows from one property of the first eigenvalue, as on that curve it is always zero, despite of the density function $p(x)$.
This property is given in Theorem \ref{thm3.3} together with another property, which states that the sign of the first eigenvalue is always positive (negative) on the right (left) area from that curve in the boundary parameters' domain.

It is well known that the set of eigenvalues is bounded from below.
In Theorem \ref{thm3.2}, the exact lower and upper bounds are found for the lowest eigenvalue depending on boundary parameters and the density function.

\section*{Acknowledgements}

This work has been supported by KAUST (individual baseline funding).

\bibliography{References}

\section*{Appendix: Relation of $S(p, \alpha, \beta)$ and $L(q, \gamma, \delta)$}\label{sec5}
Let the density function be $p \in C^2 [0,1]$. Then we can use the Liouville transformation (see e.g.~\cite{Shen:2005, Shen:2005-2}):
\begin{gather*}
  s(x)  = \dfrac{\pi}{c}\int_{0}^{x} \sqrt{p(t)} dt,\,\,\,
  y(s) = \dfrac{\pi^2}{c^2} p^{1/4}(x)u(x),  \\
  q(s) = p^{-1/4}(x)\frac{d^{2}}{ds^{2}}(p^{1/4}(x)),
\end{gather*}
where $c=\int_{0}^{1} \sqrt{p(t)} dt$, to transform $S(p, \alpha, \beta)$ into $L(q, \gamma, \delta)$.
Please note that $s(0) = 0$ and $s(1) = \pi$, therefore $s\in[0, \pi]$.
For boundary parameters we obtain
\begin{equation*}
\left\{
\begin{array}{cc}
    \cot \alpha = \dfrac{c}{\pi} \left( \dfrac{\cot \gamma}{p^{1/2}(0)} + \dfrac{p'(0)}{4 p^{3/2}(0)} \right), & \textrm{if} \; \gamma \in (0, \pi),\\
    \alpha = \pi, & \mbox{if} \; \gamma = \pi,
\end{array}\right.
\end{equation*}
\begin{equation*}
\left\{
\begin{array}{cc}
  \beta = 0, & \textrm{if} \; \delta = 0, \\
  \cot \beta = \dfrac{c}{\pi} \left( \dfrac{\cot{\delta}}{p^{1/2}(1)} + \dfrac{p'(1)}{4 p^{3/2}(1)} \right), & \textrm{if} \; \delta \in (0, \pi),
\end{array} \right.
\end{equation*}
and for the eigenvalues
\begin{equation*}
\lambda_n(p, \alpha, \beta) = \dfrac{\pi^2}{c^2} \mu_n(q, \gamma, \delta) .
\end{equation*}
When we take $p(x) \equiv 1$, then $c = 1$ and $q(x) \equiv 0$, thus $\lambda_n(1, \alpha, \beta) = \pi^2 \mu_n(0, \gamma, \delta)$. The latter relation is demonstrated in Figure \ref{fig:first}.

\end{document}